\documentclass[12pt, reqno]{amsart}
\usepackage{amsmath, amsthm, amscd, amsfonts, amssymb, graphicx, color, float}
\usepackage[bookmarksnumbered, colorlinks, plainpages]{hyperref}
\input{mathrsfs.sty}
\hypersetup{colorlinks=true,linkcolor=red, anchorcolor=green, citecolor=cyan, urlcolor=red, filecolor=magenta, pdftoolbar=true}

\newtheorem{theorem}{Theorem}[section]
\newtheorem{lemma}[theorem]{Lemma}

\theoremstyle{definition}
\newtheorem{definition}[theorem]{Definition}
\newtheorem{example}[theorem]{Example}

\theoremstyle{remark}
\newtheorem{remark}[theorem]{Remark}
\numberwithin{equation}{section}

\usepackage[T1]{fontenc}

\let\<\langle
\def\ch{\raise 0.5ex \hbox{$\chi$}}

\let\\\cr

\let\epsilon\varepsilon

\begin{document}

\title[Convergence of Random Products of Projections]{Convergence of Random Products of Countably Infinitely Many Projections}

\author[R. Eskandari]{Rasoul Eskandari}

\address{Department of Mathematics Education, Farhangian University, P.O. Box 14665-889, Tehran, Iran.}
\email{Rasul.eskandari@cfu.ac.ir, eskandarirasoul@yahoo.com}

\author[M. S. Moslehian]{Mohammad Sal Moslehian$^*$}

\address{Department of Pure Mathematics, Faculty of Mathematical Sciences, Ferdowsi University of Mashhad, P. O. Box 1159, Mashhad 91775, Iran}
\email{moslehian@um.ac.ir; moslehian@yahoo.com}

\renewcommand{\subjclassname}{\textup{2020} Mathematics Subject Classification}
 \thanks{$^*$Corresponding author}

\begin{abstract}
	
Let $r \in \mathbb{N}\cup\{\infty\}$ be a fixed number and let $P_j\,\, (1 \leq j\leq r )$ be the projection onto the closed subspace $\mathcal{M}_j$ of $\mathscr{H}$. We are interested in studying the sequence $P_{i_1}, P_{i_2}, \ldots \in\{P_1, \ldots, P_r\}$. A significant problem is to demonstrate conditions under which the sequence $\{P_{i_n}\cdots P_{i_2}P_{i_1}x\}_{n=1}^\infty$ converges strongly or weakly to $Px$ for any $x\in\mathscr{H}$, where $P$ is the projection onto the intersection $\mathcal{M}=\mathcal{M}_1\cap \ldots \cap \mathcal{M}_r$.
Several mathematicians have presented their insights on this matter since von Neumann established his result in the case of $r=2$. In this paper, we give an affirmative answer to a question posed by M. Sakai. We present a result concerning random products of countably infinitely many projections (the case $r=\infty$) incorporating the notion of pseudo-periodic function. 
	
\end{abstract}

\maketitle


\section{Introduction}
Throughout this note, let $\mathbb{B}(\mathscr{H})$ stand for the algebra of all bounded linear operators acting on a Hilbert space $(\mathscr{H},\langle \cdot, \cdot \rangle)$. The identity operator is denoted by $I$. By an (orthogonal) projection we mean an operator $P\in\mathbb{B}(\mathscr{H})$ such that $P^2=P=P^*$. As usual, $\mathscr{M}^\perp$ stands for the orthogonal complement of a (closed) subspace $\mathscr{M}$. The range and kernel of any operator $T$ are denoted by $\mathcal{R}(T)$ and $\mathcal{N}(T)$, respectively. In this context, we examine the convergence of a sequence $\{x_n\}$ in $\mathscr{H}$ strongly (that is, in norm) and weakly (that is, there exists some $x\in\mathscr{H}$ such that $\{\langle x_n,y\rangle\}$ converges to $\langle x,y\rangle$ for all $y\in\mathscr{H}$).

Let $r \in \mathbb{N}\cup\{\infty\}$ be a fixed number and let $P_j\,\, (1 \leq j\leq r )$ be the projection onto a closed subspace $\mathscr{M}_j$ of $\mathscr{H}$. Consider the sequence $P_{i_1}, P_{i_2}, \ldots \in\{P_1, \ldots, P_r\}$. An interesting problem is posed as follows:

\textbf{Problem.} Under what conditions does the sequence $\{P_{i_n}\cdots P_{i_2}P_{i_1}x\}_{n=1}^\infty$ strongly or weakly converge to $Px$ for any vector $x\in\mathscr{H}$, where $P$ denotes the projection onto the intersection $\mathscr{M}=\mathscr{M}_1\cap \ldots \cap \mathscr{M}_r$?

For $r=3$, the sequence of iterates defined as $x_n=P_{i_n}x_{n-1}$ is illustrated in Figure \ref{fig1}.

\begin{figure}[h!]
	\centering
	\begin{tabular}{c}
		\includegraphics[width=0.7\textwidth]{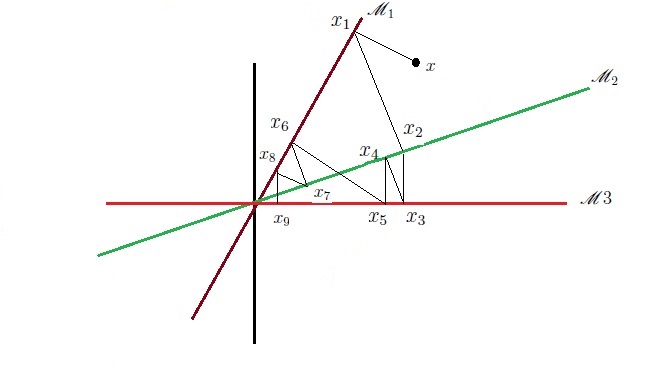}
	\end{tabular}
	\caption{}\label{fig1}
\end{figure}

This problem is inherently complicated, in particular, in the case where $r=\infty$; {\color{blue}see \cite{DYE, DYE2, SAK}}. 

{\color{blue} When $r\in \mathbb{N}$, Amemiya} and Ando \cite{AA} proved that the sequence $\{P_{i_n}\cdots P_{i_2}P_{i_1}x\}_{n=1}^\infty$ converges weakly to $Px$ for any vector $x\in\mathscr{H}$. They conjectured that their result is also valid if we replace ``weakly'' with ``strongly". 

For $r\geq 3$, the conjecture was already true if $\mathscr{H}$ is finite-dimensional; as shown in Pr\'{a}ger \cite{PRA}. 

Halperin \cite[Theorem 1]{N} established that if $Q_i$'s are chosen in such a way that $P_{i_{rn}}\cdots P_{i_2}P_{i_1}=(P_r\cdots P_2P_1)^{n}$ for all $n$, then the subsequence $\{P_{i_{rn}}\cdots P_{i_2}P_{i_1}x\}_{n=1}^\infty$ of $\{P_{i_n}\cdots P_{i_2}P_{i_1}x\}_{n=1}^\infty$ strongly converges to $Px$. However, the convergence $$\lim_n\|(P_1\cdots P_n)^nx-Px\|=0$$ may not be uniform on bounded subsets of initial points $x$. Many mathematicians are working on the rate of convergence; see \cite{BAD} and references therein. Sakai \cite{SAK} extended the Halperin result for quasi-periodic functions and a finite number of projections. He then posed a question of whether his results are still valid for a larger class than quasi-periodic functions or countably infinitely many projections. In \cite{MDD}, the authors addressed Sakai's question for a broader class beyond quasi-periodic functions, which are called {\color{blue}quasi-periodic} sequences; {\color{blue}see \cite{THI, THI2}}. In this paper, we explore the convergence of non-periodic products of projections, a topic which is also examined in various works {\color{blue}like \cite{PUS, PUS2}}.

The case $r=2$ was proved much earlier by von Neumann \cite{VON}. Indeed, he {\color{blue}established} his {\color{blue}well known} alternating projection theorem by showing that if $P_1$ and $P_2$ are projections, then the sequence $P_1x, P_2P_1x, P_1P_2P_1x, \ldots$ converges strongly to $Px$ in which $P$ is the projection onto the intersection of their corresponding closed subspaces. {\color{blue}Simple geometric proofs of von Neumann's theorem were obtained by  Kopeck\'{a} and Reich in \cite{KR1, KR2}}.

However, Paszkiewicz \cite{PAS} (with $r=5$) as well as Kopeck\'{a} and V. M\"{u}ller \cite{KM} (with $r=3$) gave counterexamples for the Amemiya--Ando conjecture. In general, the behavior of projections can be strange as shown by Kopeck\'{a} and Paszkiewicz \cite{KP}. If $\mathscr{H}$ is an infinite-dimensional Hilbert space, then there exist three projections $P_1, P_2$, and $P_3$ onto closed subspaces of $\mathscr{H}$ such that for every nonzero vector $x\in\mathscr{H}$, there exist integers $k_1,k_2, \ldots \in \{1,2,3\}$ such that $\{P_{k_n},\cdots P_{k_2}P_{k_1}x\}$ does not converge strongly.

Variations of this problem have applications in multiple disciplines such as partial differential equations, approximation theory, biomathematics, and computer science; see \cite{DEU} and references therein.

In this paper, we provide a positive answer to Sakai's question by presenting a result concerning random products of an infinite number of projections that involve a general notion of quasi-periodic functions, namely pseudo-periodic functions, which differ from {\color{blue} quasi-periodic functions}. We hope that this insight can contribute to solving the problem. We also provide concrete examples to clarify our results.


\section{Main results}

We start our work with recalling the following notion {\color{blue}appearing in \cite[p. 206]{DYE2} and \cite{SAK}}.

\begin{definition}\label{qp}
Let $r\in \mathbb{N}$. A function $\sigma: \mathbb{N} \to \{1, 2, \ldots, r\}$ is called \emph{quasi-periodic} if there exists an integer $m\geq r$ such that for each integer $k\geq 0$, it holds that
	\begin{align*}
		\{\sigma(k+1), \sigma(k+2), \ldots, \sigma(k+m)\}= \{1, 2, \ldots, r\},
	\end{align*}
	or equivalently, for each $j$, the inverse image of $\{j\}$ under $\sigma$ is an infinite set, and if
	$\{l_{n}\}_{n=1}^\infty$ is the increasing sequence of all natural numbers such that $\sigma(l_{n})=j$, then
	\begin{align*}
		I(\sigma,j):=\sup_n \left(l_{n}-l_{n-1}\right)<\infty\,,
	\end{align*}
	where $l_{0}=0$.
\end{definition}

The following is an extension of Halperin's result.
\begin{theorem}\cite[Theorem]{SAK}\label{sakamain}
	Let $r$ be a positive integer, let $P_1,\ldots, P_r\in \mathbb{B}(\mathscr{H})$ be projections onto closed subspaces $\mathscr{M}_1,\ldots, \mathscr{M}_r$ of $\mathscr{H}$, respectively, and let $P$ be the projection onto $\bigcap_{i=1}^r\mathscr{M}_i$ . Let $\sigma:\mathbb{N}\to\{1,\ldots, r\}$ be quasi-periodic and set $T_1:=P_{\sigma(1)}$ and $T_n:=P_{\sigma(n)}T_{n-1}$. Then, $\{T_nx\}$ strongly converges to $Px$. Furthermore,
	\[
	\|T_nx-T_mx\|^2\leq \big((b-1)(b-2)+3\big)\sum_{k=m}^{n-1}\|T_{k+1}x-T_kx\|^2\,,
	\]
	where $x\in \mathscr{H}$, $b=\max_{1\leq j\leq r}I(\sigma,j)$ and $n> m\geq 1$. 
\end{theorem}

Now, we extend the notion of a quasi-periodic function. 

Suppose that $\sigma:\mathbb{N}\to \mathbb{N}$ is a function such that for each $j\in \mathbb{N}$, {\color{blue} the inverse image of $\{j\}\subseteq\sigma(\mathbb{N})$ under $\sigma$ is an infinite set}, and let
$\{l_{n}\}_{n=1}^\infty$ be the increasing sequence (depending on $j$) of all
natural numbers such that $\sigma(l_{n})=j$ and $l_{0}=0$.
Let $\Gamma_F$ denote the set of all $j$ such that
\begin{align}\label{sup}
I(\sigma,j)=\sup_n \left(l_{n}-l_{n-1}\right)<\infty.
\end{align}
Set $\Gamma_\infty:=\sigma(\mathbb{N})\backslash \Gamma_F$. Hence, $\Gamma_F$ and $\Gamma_\infty$ are disjoint subsets of $\mathbb{N}$ and {\color{blue}  $\Gamma_F\cup\Gamma_\infty=\sigma(\mathbb{N})$. If $\Gamma_F=\sigma(\mathbb{N})$, then $\sigma$ is a quasi-periodic function.}
Let $\{k_n\}$ denote the complement of the union of all sequences $\{l_n\}$ in the form of an increasing sequence. Evidently, $\sigma(k_n)\in \Gamma_\infty$. 
\begin{definition} \label{pse}
	A function $\sigma:\mathbb{N}\to\mathbb{N}$ is a \emph{pseudo-periodic function} if $\Gamma_F=\{1, 2, \ldots, r\}$ for some positive integer $r$ and $\{k_n-k_{n-1}\}$ is an increasing sequence. 
In this case, there exists an integer $m\geq r$ such that for each integer $k\geq 0$, 
\begin{align*}
	\{1, 2, \ldots, r\}\subseteq\{\sigma(k+1), \sigma(k+2), \ldots, \sigma(k+m)\}.
\end{align*}

It should be noted that the terms {\color{blue}``quasi-periodic''}, ``pseudo-periodic'', and ``almost-periodic'' are established in the context of functions defined on the real line. However, their usage in the context of sequences appears to differ from their traditional meanings.
\end{definition}

The following is a typical example of random functions $\sigma:\mathbb{N}\to\mathbb{N}$ we study.
\begin{example}\label{exam}
Let $P_1$, $P_2$, and $P_3$ be arbitrary projections on a Hilbert space $\mathscr{H}$. Let $\{P_i\}_{i=4}^\infty$ be monotonically decreasing projections. Let $\sigma':\mathbb{N}\to \{4,5,\ldots\}$ be a random function {\color{blue} such that the inverse image of $\{j\}$ is an infinite set for each $j\geq 4$.} Define $\sigma:\mathbb{N}\to \mathbb{N}$ as follows: 
\[
\sigma(n)=\begin{cases}
1&n=3k-1\\
2&n=3k-2\\
3&n=3k,n\neq 3^{k'}\\
\sigma'(k')&n=3k,n=3^{k'}\\
\end{cases}
\qquad \mbox{for~some~ $k,k'\in\mathbb{N}$}
\]
For instance, when $j=3$, the sequence of indices for which $\sigma(l_n)=3$ is $\{l_n\}=6, 12, 15, 18, \ldots$.

It is easy to observe that 
\[I(\sigma,1)=I(\sigma,2)=3,\quad I(\sigma,3)=6, \quad I(\sigma,4)=I(\sigma,5)= \cdots=\infty\,.
\]
This shows that $\sigma$ is not quasi-periodic. 
Here, we ar examining the sequence
\begin{align*} 
	T_n &:=P_{\sigma(n)}\cdots P_{\sigma(4)} P_{\sigma(3)}P_{\sigma(2)} P_{\sigma(1)} \\
	&=P_{\sigma(n)}\cdots~ P_6P_1 \hspace{0.35cm} \cdots \hspace{0.15cm} P_3P_1P_2P_3P_1P_2P_3P_1P_2P_5P_1P_2P_3P_1P_2P_4P_1P_2\\
	&=P_{\sigma(n)}\cdots P_{\sigma(k_3)}P_1 \cdots P_3P_1P_2P_3P_1P_2P_3P_1P_2P_{\sigma(k_2)}P_1P_2P_3P_1P_2P_{\sigma(k_1)}P_1P_2
\end{align*}

The sequence $k_1, k_2, k_3, \ldots$ is $3, 9, 27, \ldots$. Evidently, $\sigma$ is pseudo-periodic. 
\end{example}


To achieve our result, we need several key lemmas.


{\color{blue}\begin{lemma}\label{sakai}
	 Let $P_n\in \mathbb{B}(\mathscr{H})$ be {\color{blue} the projection} onto a closed subspace $\mathscr{M}_n$ of $\mathscr{H}$ for each $n\in \mathbb{N}$. Let $\sigma:\mathbb{N}\to\mathbb{N}$ be a random function and set $T_1:=P_{\sigma(1)}$ and $T_n:=P_{\sigma(n)}T_{n-1}$. Then 
	\[
	\lim_{n\to \infty}\|T_{n-k}x-T_nx\|=0
	\]
	for each $k\geq 1$. 
\end{lemma}} 
\begin{proof}
	The sequence $\{\|T_nx\|\}$ is decreasing since
	\[
	\|T_{n-1}x\|\geq \|T_{n}x\|\geq 0\qquad(n\geq 2).
	\]
	Hence, 
	\begin{align}\label{conv}
		\mbox{The sequence~} \{\|T_nx\|\} \mbox{~is a convergent sequence.}
	\end{align}
 
	For each $n$, we have 
	\begin{align}\label{nn+1}
		\|T_{n-1}x-T_nx\|^2&=\|(I-P_{\sigma(n)})T_{n-1}x\|^2 \nonumber\\
		&=\langle T_{n-1}x-P_{\sigma(n)}T_{n-1}x, T_{n-1}x-P_{\sigma(n)}T_{n-1}x\rangle \nonumber\\
		&=\|T_{n-1}x\|^2-\|P_{\sigma(n)}T_{n-1}x\|^2 \nonumber\\
		&=\|T_{n-1}x\|^2-\|T_{n}x\|^2.
	\end{align}
	Hence,
	\begin{equation}\label{eq01}
		\lim_{n\to \infty}\|T_{n-1}x-T_nx\|=0.
	\end{equation}
	By induction on $k$ and using the triangle inequality, we can obtain the required result. 
\end{proof}


\begin{lemma}\label{lemma lim}
	With the notation of Definition \ref{pse}, let $i\in \mathbb{N}$. For the unique integers $k_n$ and $k_{n+1}$ depending on $i$ such that $k_n\leq i<k_{n+1}$, it holds that 
	\[
	\lim_{i\to \infty}\|T_ix-T_{k_{n+1}}x\|=\lim_{i\to \infty}\|T_ix-T_{k_{n}}x\|=0\,.
	\]
\end{lemma}
\begin{proof}
	Consider the set $\mathscr{S}=\{T_{k_n+1}x, T_{k_n+2}x, \ldots, T_{k_{n+1}-1}x\}$. If we set $T_{k_n}x=y$, then
	\[
	\mathscr{S}=\{P_{k_n+1}y,P_{k_n+2}P_{k_n+1}y\,,\ldots\,, P_{k_{n+1}-1}\ldots P_{k_n+2}P_{k_n+1}y \}
	\]
	In fact, $\mathscr{S}$ is constructed from a part of a quasi-periodic sequence, since $\{k_n+1, k_n+2, \ldots, k_{n+1}-1\} \in \Gamma_F$. Hence, Theorem \ref{sakamain} ensures that
	\begin{align}\label{msm18}
	\|T_jx-T_ix\|^2\leq M\sum_{k=i}^{j-1}\|T_{k+1}x-T_{k}x\|^2,
	\end{align}
	for a fixed number $M\geq 0$ and all $j$ satisfying {\color{blue} $k_{n}<i < j < k_{n+1}$.} 
	From equality \eqref{nn+1}, we have 
	\[
	\|T_{k+1}x-T_{k}x\|^2=\|T_{k}x\|^2-\|T_{k+1}x\|^2\,,
	\]
	for each $ k$.
	 Therefore, inequality \eqref{msm18} yields that
	 \[
	 \|T_jx-T_ix\|^2\leq M(\|T_ix\|^2-\|T_jx\|^2)\,,\qquad{\color{blue}(k_{n}< i < j < k_{n+1}).}
	 \]

 Utilizing the parallelogram law and Lemma \ref{sakai}, we have 
 
	\begin{align*}
		\lim_i\|T_{k_{n+1}}x-T_ix\|^2&\leq \lim_i2\|T_{k_{n+1}}x-T_{k_{n+1}-1}x\|^2+2\|T_{k_{n+1}-1}x-T_ix\|^2\\
		&\leq \lim_i2\|T_{k_{n+1}}x-T_{k_{n+1}-1}x\|^2+\lim_i 2M(\|T_ix\|^2-\|T_{k_{n+1}-1}x\|^2)\\
		&=0,
	\end{align*}
	Since as $i\to \infty$, we have $k_{n+1}\to \infty$.
	By the same reasoning we get 
	\[
	\lim_{i\to \infty}\|T_ix-T_{k_{n}}x\|=0\,.
	\]
\end{proof}

\begin{lemma}\label{lemma p}
Let $r$ be a positive integer and let $P_1,\ldots, P_r\in \mathbb{B}(\mathscr{H})$ be projections onto closed subspaces $\mathscr{M}_1,\ldots, \mathscr{M}_r$ of $\mathscr{H}$, respectively. Let $\{P_i\}_{i=r+1}^\infty$ be monotonically decreasing projections on $\{\mathscr{M}_i\}_{i=r+1}^\infty$. {\color{blue} Let $\sigma:\mathbb{N}\to\mathbb{N}$ be {\color{blue} a pseudo-periodic function}. Let $T_1:=P_{\sigma(1)}$ and $T_n:=P_{\sigma(n)}T_{n-1}$. If $x\in\mathscr{H}$ is arbitrary, then 
\[
\lim_{n\to\infty}\|T_{k_n-1}x-P_{r+1}T_{k_n-1}x\|=0\,.
\]}
\end{lemma}
\begin{proof}
	{\color{blue} Suppose that $x\in \mathscr H$ is arbitrary. Since, $P_{r+1}\geq P_{k_n}$ for all $n\geq 1$ we have
	\begin{align*}
\|T_{k_n-1}x-P_{r+1}T_{k_n-1}x\|^2&=\|(I-P_{r+1})T_{k_n-1}x\|^2\\
&=\|T_{n_k-1}x\|^2-\|P_{r+1}T_{k_n-1}x\|^2\\
&= \|T_{n_k-1}x\|^2-\langle P_{r+1}T_{k_n-1}x,P_{r+1}T_{k_n-1}x\rangle\\
&\leq  \|T_{n_k-1}x\|^2-\langle P_{k_n}T_{k_n-1}x,P_{k_n}T_{k_n-1}x\rangle\\
&=\|T_{k_n-1}x\|^2-\|T_{k_n}x\|^2.
	\end{align*}
Now, the  result follows immediately from Lemma \ref{sakai}.}
\end{proof}



We establish our first result concerning the weak convergence of a random product of projections that involved a pseudo-periodic function.

\begin{theorem}\label{mainth}
	Let $r$ be a positive integer and let $P_1,\ldots, P_r\in \mathbb{B}(\mathscr{H})$ be distinct projections. Let $\{P_i\}_{i=r+1}^\infty$ be a monotonically decreasing sequence of (not necessarily distinct) projections. Let $\mathcal{R}(P_i)=\mathscr{M}_i$ for each $i\in\mathbb{N}$. 
	
	Let $\sigma:\mathbb{N}\to\mathbb{N}$ be {\color{blue} a pseudo-periodic function.} Set $T_1:=P_{\sigma(1)}$ and $T_n:=P_{\sigma(n)}T_{n-1}$. Then the sequence $\{T_nx\}$ converges to $Px$ weakly, where $P$ is the projection onto $\bigcap_{i=1}^\infty\mathscr{M}_i$.
\end{theorem}
\begin{proof}
Without loss of generality, we assume that $\Gamma_F=\{1,\ldots, r\}$, $\Gamma_\infty=\{r+1, r+2, \ldots \}$.  {\color{blue} Since $\{P_i\}_{i=r+1}^\infty$ is a monotonically decreasing sequence, we can assume that $\{k_n\}$ satisfies $k_{n+1}-k_n>1$.}  With the notation in the definition of the pseudo-periodic function $\sigma$, $P_{\sigma(k_n)}$ appears as the $k_n$-th projection from the right side in the definition of $T_n:=P_{\sigma(n)}\cdots P_{\sigma(4)} P_{\sigma(3)}P_{\sigma(2)} P_{\sigma(1)}$. Hence, we can write 
 	\begin{align*}
 		T_{k_n+1}&= P_{k_n+1}P_{k_n}T_{k_n-1}.
 \end{align*}
 {\color{blue}Thus, $k_n+1\in\{1, 2, \ldots, r\}$. 
 Let $x, y\in \mathscr{H}$ be arbitrary. Using the decomposition  $x=x_{1}+x_{2}$, where
	$$x_{1} \in \mathcal{R}(I-P_{r+1})\quad \textrm{and}\quad x_{2}\in  \mathscr{M}_{r+1}\,,$$
	we have 
	\begin{align*}
	\langle T_{k_n-1}y, x\rangle &=\langle T_{k_n-1}y, x_{1}+x_{2}\rangle\\
	&=\langle T_{k_n-1}y,(I-P_{r+1})z\rangle+\langle T_{k_n-1}y,x_{2}\rangle,\qquad(\mathrm{for~some~} z\in \mathscr{H})\\
	&=\langle (I-P_{r+1})T_{k_n-1}y, z\rangle+\langle y,T_{k_n-1}x_{2}\rangle \\
	&=\langle (T_{k_n-1}-P_{r+1}T_{k_n-1})y,z\rangle+\langle y,Px\rangle\\
		&=\langle (T_{k_n-1}-P_{r+1}T_{k_n-1})y,z\rangle+\langle Py,x\rangle.\\
		\end{align*}
		Hence, 
		\begin{align*}
\left|\langle T_{k_n-1}y, x\rangle-\langle Py,x\rangle\right|\leq	\|T_{k_n-1}y-P_{r+1}T_{k_n-1}y\|\|z\|.
\end{align*}}
 {\color{blue}From Lemma \ref{lemma p}} we derive that $\{T_{k_n-1}x\}$ weakly converges to $Px$. Hence, $\{T_{k_n}x\}$ also weakly converges to $Px$. For each $i$, there exists a positive integer $n$ that depends on $i$, such that $k_n\leq i<k_{n+1}$. 
	
	It follows from {\color{blue}Lemma \ref{lemma lim} that  $\lim_i\|T_ix-T_{k_n}x\|=0$.} Therefore,
	\[
	\lim_i	\langle T_ix,y\rangle={\color{blue}\lim_i\left(\langle T_ix-T_{k_n}x,y\rangle +\langle T_{k_n}x,y\rangle\right)} =\langle Px,y\rangle\,.
	\]
	This implies that $\{T_ix\}$ weakly converges to $Px$.
	
\end{proof}

Additional conditions are required to guarantee the strong convergence of the sequence $\{T_nx\}$ to $Px$. Specific terminology and lemmas are needed to achieve this result. 

The following notion is introduced in \cite{BAU}.

\begin{definition}\label{defangle}
	The angle of an $r$-tuple of closed subspaces $(\mathscr{M}_1, \mathscr{M}_2,\ldots, \mathscr{M}_r)$ in a Hilbert space $\mathscr{H}$ is the angle in $[0, \frac{\pi}{2}]$
	whose cosine is defined as
	\[
	c_b(\mathscr{M}_1, \mathscr{M}_2,\ldots, \mathscr{M}_r)=\|P_r\ldots P_1P_{(\bigcap_{i=1}^r\mathscr{M}_i)^\perp}\|
	\]
\end{definition}
 
{\color{blue}
The concept of \textit{inner inclination} of an $m-$tuple $(\mathscr{M}_1,\mathscr{M}_2,\ldots,\mathscr{M}_m)$ of closed subspaces of a Hilbert space $\mathscr{H}$ is defined in \cite[Definition 2.2]{PUS3} as: 
\begin{equation}\label{eqinnerinclination}
\tilde{\ell}(\mathscr{M}_1,\mathscr{M}_2,\ldots,\mathscr{M}_m)=\min_{1\leq i\leq m}\inf_{x\in \mathscr{M}_i\backslash \mathscr{M}}\frac{\max_{1\leq j\leq m}\mathrm{dist}(x;\mathscr{M}_j)}{\mathrm{dist}(x;\mathscr{M})},
\end{equation}
where $\mathscr{M}=\bigcap_{i=1}^m\mathscr{M}_i \neq \mathscr{H}$ and the minimum is taken over all $i=1,\cdots, m$. Here, $\mathrm{dist}(x;\mathscr{M})$ denotes the distance between a point $x \in \mathscr{H}$ and a subspace $\mathscr{M}$ of $\mathscr{H}$. In addition, the concept of \emph{inclination} is defined in \cite{BAD} by:
\begin{equation}\label{eqinclination}
	\ell(\mathscr{M}_1,\mathscr{M}_2,\ldots,\mathscr{M}_m)=\inf_{x\notin \mathscr{M}}\frac{\max_{1\leq j\leq m}\mathrm{dist}(x;\mathscr{M}_j)}{\mathrm{dist}(x;\mathscr{M})},
\end{equation}

The following example demonstrates that $c_b(\mathscr{M}_1, \mathscr{M}_2,\ldots, \mathscr{M}_m)$ differs from both the quantities $\tilde{\ell}(\mathscr{M}_1,\mathscr{M}_2,\ldots,\mathscr{M}_m)$ and $\ell(\mathscr{M}_1,\mathscr{M}_2,\ldots,\mathscr{M}_m)$, in general.

\begin{example}
	Let $\mathscr H=\mathbb{C}^2$ with the canonical basis $\{e_1, e_2\}$. Suppose that 
	\[
	\mathscr{M}_1=<e_1>\quad\mbox{and}\quad\mathscr{M}_2=<e_2>.
	\]
	Then $\mathscr{M}=\mathscr{M}_1\cap\mathscr{M}_2=\{0\}$. Let $P_1, P_2$, and $P=0$ be the projections onto $\mathscr{M}_1$, $\mathscr{M}_2$, and $\mathscr{M}$, respectively.	It follows from Definition \ref{defangle} that 
	\[
	c_b(\mathscr{M}_1, \mathscr{M}_2)=\|P_1P_2(I-P)\|=0.
	\]
	Let $x=\alpha_1e_1+\alpha_2e_2 \in \mathscr{H}\backslash\{0\}$. Therefore,
	\[
	\mathrm{dist}(x;\mathscr{M}_1)=\|x-P_1x\|=|\alpha_2|,\,\, \mathrm{dist}(x;\mathscr{M}_2)=|\alpha_1|,\,\,\mbox{and}\,\,\mathrm{dist}(x;\mathscr{M})=\sqrt{|\alpha_1|^2+|\alpha_2|^2}.
	\]
	Since $\sqrt{|\alpha_1|^2+|\alpha_2|^2} \leq 2\max\{|\alpha_1,\alpha_2|\}$, we have
	\[\ell(\mathscr{M}_1,\mathscr{M}_2)=\inf_{x\notin \mathscr{M}}\frac{\max_{1\leq j\leq 2}\mathrm{dist}(x;\mathscr{M}_j)}{\mathrm{dist}(x;\mathscr{M})}=\inf_{x=\alpha_1e_1+\alpha_2e_2\neq 0}\frac{\max\{|\alpha_1|,|\alpha_2|\}}{\sqrt{|\alpha_1|^2+|\alpha_2|^2}} \geq \frac{1}{2}.\]
	Moreover,
	$$\inf_{x\in \mathscr{M}_1\backslash \mathscr{M}}\frac{\max\{\mathrm{dist}(x;\mathscr{M}_1),\mathrm{dist}(x;\mathscr{M}_2)\}}{\mathrm{dist}(x;\mathscr{M})}=\inf_{x\in \mathscr{M}_1\backslash \mathscr{M}}\frac{\max\{0,\|x\|\}}{\|x\|}=1$$
	and
		$$\inf_{x\in \mathscr{M}_2\backslash \mathscr{M}}\frac{\max\{\mathrm{dist}(x;\mathscr{M}_1),\mathrm{dist}(x;\mathscr{M}_2)\}}{\mathrm{dist}(x;\mathscr{M})}=\inf_{x\in \mathscr{M}_1\backslash \mathscr{M}}\frac{\max\{\|x\|,0\}}{\|x\|}=1.$$
		This yields that $\tilde{\ell}(\mathscr{M}_1,\mathscr{M}_2)=1$.
\end{example}}

The following lemma is interesting on its own right.

\begin{lemma}\label{lemmclsed} \cite{BAU}
\[
c_b(\mathscr{M}_1, \mathscr{M}_2,\ldots, \mathscr{M}_r)<1
\]
if and only if
\[
\mathscr{M}_1^\perp+\mathscr{M}_2^\perp\ldots+\mathscr{M}_r^\perp\,.
\]
is closed.
\end{lemma}

\begin{remark}\label{remark1}
In view of the above lemma, if $c_b(\mathscr{M}_1, \mathscr{M}_2,\ldots, \mathscr{M}_r)<1$, then for each $t>r$ we have 
\[
c_b(\mathscr N_1, \mathscr N_2,\ldots, \mathscr N_t)<1\,,
\]
when $\{\mathscr N_i: 1\leq i\leq t\}=\{\mathscr{M}_i:1\leq i\leq r\}$.
\end{remark}

The proof of the subsequent lemma is straightforward, therefore, we omit it.

\begin{lemma}\label{lemma sakae}
	Let $Q$ be the projection onto a closed subspace of a
	Hilbert space. Let $x$ and $y$ be elements in the space. Then
	\begin{equation}
		\|x-y\|^2\leq \|x-Qy\|^2+\|x-Qx\|^2+2\|y-Qy\|^2.
	\end{equation}
\end{lemma}

The next lemma {\color{blue}reads} as follows.

\begin{lemma}\label{lemma subsequence}
Let $\{P_i\}_{i\in J}$ be a finite or infinite sequence of projections acting on a Hilbert space $\mathscr{H}$. 	Let $\sigma:\mathbb{N}\to J$ be a random function. Set $T_1:=P_{\sigma(1)}$ and $T_n:=P_{\sigma(n)}T_{n-1}$. {\color{blue}Let the sequence $\{T_nx\}$ weakly converge to $Px$}, where $P$ is the projection onto $\bigcap_{i\in J}\mathcal{R}(P_i)$. {\color{blue}If there exists a subsequence $\{T_{n_kx}\}$ which converges to $Px$ strongly}, then the sequence $\{T_nx\}$ converges to $Px$ strongly.
\end{lemma}
\begin{proof}
{\color{blue}Let  $|T|=(T^*T)^{1/2}$. Since 
\begin{align*}
	\langle |T_n|^2x,x\rangle&=\langle P_{\sigma(n)}T_{n-1}x,P_{\sigma(n)}T_{n-1}x\rangle=\langle P_{\sigma(n)}T_{n-1}x, T_{n-1}x\rangle\\
	&\leq \langle T_{n-1}x,T_{n-1}x\rangle= \langle |T_{n-1}|^2x, x\rangle
\end{align*}
we have $|T_{n-1}|^2\geq |T_n|^2$, and hence, $|T_{n-1}|\geq |T_n| \geq 0$. Therefore, by the Vigier theorem \cite[Theorem 4.1.1]{MUR}, there exists a positive operator $S$ such that the sequence $\{|T_n|x\}$ strongly converges to $Sx$ for all $x$. Thus, $\|T_nx\|^2=\|\,|T_n|x\|^2 \to \|Sx\|^2$ as $n \to \infty$. Therefore,
\begin{align}\label{slim}
\lim_{n\to\infty}\|T_nx\|=\|Sx\|\,.
\end{align}
  On the other hand, $\lim_{n\to \infty}\|T_{n_k}x\|=\|Px\|$, which ensures that $\|Px\|=\|Sx\|$. {\color{blue}Now, the weakly convergence of} $\{T_nx\}$ and \eqref{slim} complete the proof.}
\end{proof}

{\color{blue}\begin{remark}
For a finite set $J$, Amemiya and Ando \cite{AA} proved that the sequence $\{T_nx\}$ converges weakly to $Px$. For an infinite $J$, additional conditions may be required; see \cite{PUS}.
\end{remark}}

Our next main result is as follows.

\begin{theorem}\label{thstrong}
	Let $r$ be a positive integer and let $P_1,\ldots, P_r\in \mathbb{B}(\mathscr{H})$ be distinct projections such that {\color{blue}$c_b(\mathcal{R}(P_1), \ldots, \mathcal{R}(P_r))<1$}. Let $\{P_i\}_{i=r+1}^\infty$ be a monotonically decreasing sequence of (not necessarily distinct) projections such that $\bigcap_{i=1}^r\mathcal{R}(P_i)=\bigcap_{i=1}^\infty\mathcal{R}(P_i)$.
	Let $\sigma:\mathbb{N}\to\mathbb{N}$ be pseudo-periodic. Set $T_1:=P_{\sigma(1)}$ and $T_n:=P_{\sigma(n)}T_{n-1}$. Then the sequence $\{T_nx\}$ strongly converges to $Px$, where $P$ is the projection onto $\bigcap_{i=1}^\infty\mathcal{R}(P_i)$. 	
\end{theorem}
\begin{proof}
Let $x\in\mathscr{H}$. According to Theorem \ref{mainth}, the sequence $\{T_nx\}$ weakly converges to $Px$. Fix $i_0<j_0$ be positive integers. For each positive integer $i$ with $i_0\leq i \leq j_0$ and each positive integer $t$ (depending on $i$) with {\color{blue} $k_{i}\leq t \leq k_{i+1}$, Lemma \ref{lemma sakae} with $Q=P_{t+1}$ ensures that 
\begin{align}\label{eq000}
	\|Px-T_{t}x\|^2&\leq\|Px-T_{t+1}x\|^2+\|Px-P_{t+1}Px\|^2+2\|T_{t}x-T_{t+1}x\|^2 \nonumber\\
	&=\|Px-T_{t+1}x\|^2+2\|T_{t}x-T_{t+1}x\|^2\,.
\end{align}
since $P_jP=PP_j=P$ for all $j$.
Adding inequalities \eqref{eq000} for all $k_{i}\leq t \leq k_{i+1}-1$, gives us:}
\begin{align}\label{mox5}
	\|Px-T_{k_{i}}x\|^2\leq\|Px-T_{k_{i+1}}x\|^2+2\sum_{t=k_{i}}^{k_{i+1}-1}\|T_{t}x-T_{t+1}x\|^2 
\end{align}
for each fixed $i_0\leq i \leq j_0$. Summing up inequalities \eqref{mox5} over all $i_0\leq i \leq j_0$, we reach

\begin{align}\label{eq0000}
\|Px-T_{k_{i_0}}x\|^2&\leq\|Px-T_{k_{j_0+1}}x\|^2+2\sum_{t=k_{i_0}}^{k_{{j_0+1}}-1}\|T_{t}x-T_{t+1}x\|^2\nonumber\\
&{\color{blue}=\|Px-T_{k_{j_0+1}}x\|^2+2\sum_{t=k_{i_0}}^{k_{{j_0+1}}-1}(\|T_{t}x\|^2-\|T_{t+1}x\|^2)}\nonumber\\
&= \|Px-T_{k_{j_0+1}}x\|^2+2(\|T_{k_{i_0}}x\|^2-\|T_{k_{j_0+1}}x\|^2)\nonumber\\
&\leq 2\|Px-T_{k_{j_0+1}-1}x\|^2+2\|T_{k_{j_0+1}}x-T_{k_{j_0+1}-1}x\|^2\nonumber\\
& \quad +2(\|T_{k_{i_0}}x\|^2-\|T_{k_{j_0}}x\|^2),
\end{align}
where we use the parallelogram law to get the last inequality.

Note that $\|P_r \cdots P_1 (I-P)\|<1$ and $Q_1\cdots Q_m$ is a product of $m\geq r$ projections such that $\{Q_1, \ldots, Q_m\}=\{P_1,\ldots, P_r\}$. Remark \ref{remark1} implies that $\|Q_1\cdots Q_m (I-P)\|<1$. 
Hence, there exists a constant $c<1$ such that $\|P_{\sigma(n+m)}\ldots P_{\sigma(n+1)}(I-P)\|<c$ whenever $k_{j_0}<n+1< \cdots <n+m<k_{j_0+1}$. Therefore,
\begin{align*}
\|Px-T_{k_{j_0+1}-1}x\|&=\|PT_{k_{j_0}}x-T_{k_{j_0+1}-1}x\|\\
&{\color{blue}= \|(P_{k_{j_0+1}-1}P_{k_{j_0+1}-2}\cdots P_{k_{j_0}+1}-P)T_{k_{j_0}}x\|}\\
&= \|(P_{k_{j_0+1}-1}P_{k_{j_0+1}-2}\cdots P_{k_{j_0}+1}({\color{blue}I}-P))T_{k_{j_0}}x\|\\
&\leq \|P_{k_{j_0+1}-1}P_{k_{j_0+1}-2}\cdots P_{k_{j_0+1}-m}({\color{blue}I}-P)\|\\
&\quad \times \|P_{k_{j_0+1}-m-1}\cdots P_{k_{j_0}-2m}({\color{blue}I}-P)\|\ldots \\
&\quad \times \|P_{k_{j_0}+s+2m}\cdots P_{k_{j_0}+s+m+1}({\color{blue}I}-P)\|\qquad(0\leq s\leq m-1 ) \\
&\quad \times \|P_{k_{j_0}+m+s}\cdots P_{k_{j_0}+1}\|\|T_{k_{j_0}}x\|\\
&\leq c^\nu\|x\|,
\end{align*}
{\color{blue}for some $\nu$, where $m$ is given in the definition of quasi-periodic. So, by \eqref{eq0000} we get 
 \begin{equation}
\|Px-T_{k_{i_0}}x\|^2\leq 2c^\nu\|x\|^2+2\|T_{k_{j_0+1}}x-T_{k_{j_0+1}-1}x\|^2 +2(\|T_{k_{i_0}}x\|^2-\|T_{k_{j_0}}x\|^2).
	\end{equation}
	Since $\{k_n-k_{n-1}\}_{n=1}^\infty$ is an increasing sequence, we have $\nu \to \infty$ if $j_0\to \infty$. Since  $0<c<1$, it  follows from  \eqref{conv} and  \eqref{eq01} that $\lim_{n\to\infty}\|T_{k_n}x-Px\|=0$. Now, the result is obtained by utilizing Lemma \ref{lemma subsequence}.}
\end{proof}

The next example illustrates Theorem \ref{thstrong}.
\begin{example}
	Let $\mathscr{H}$ be a separable Hilbert space with $\{e_i:i\in \mathbb{N}\}$ as its orthonormal basis. Let 
	\[
	\mathscr{M}_1=\mathrm{span}\{e_{2k-1}:k\in \mathbb{N}\}, \quad \mathscr{M}_2=\mathrm{span}\left\{\frac{e_{2k-1}+e_{2k}}{2}:k\in \mathbb{N}\right\},\] and
	\[\mathscr{M}_i=\mathrm{span}\{e_{3j}:j\geq i-2\}\,, 
	\]
	for $i\geq 3$. Let $P_i$ be the projection onto $\mathscr{M}_i$ for each $i\geq 1$. It is easy to verify that $\mathscr{M}_1+\mathscr{M}_2=\mathscr{H}$ is closed. Therefore, $\mathscr{M}_1^\perp +\mathscr{M}_2^\perp=(\mathscr{M}_1\cap\mathscr{M}_2)^\perp=\mathscr{H}$ is closed (see \cite[Lemma 11]{DEU1}). In addition, $\mathscr{M}_1\cap \mathscr{M}_2=\{0\}=\bigcap_{i=3}^\infty\mathscr{M}_i$. It follows from Lemma \ref{lemmclsed} that
	\[
	c_b(\mathscr{M}_1,\mathscr{M}_2)<1\,. 
	\]
	Let $P_i$ be the projection onto $\mathscr{M}_i$ for each $i\geq 1$. Let $\sigma:\mathbb{N}\to\mathbb{N}$ be pseudo-periodic with $\Gamma_F=\{1,2\}$. Set $T_1:=P_{\sigma(1)}$ and $T_n:=P_{\sigma(n)}T_{n-1}$. From Theorem \ref{thstrong} we conclude that the sequence $\{T_nx\}$ strongly converges to $0$.
\end{example}

 \medskip
\section*{Disclosure statements}

\medskip
\noindent \textit{Author Contributions Statement.} Both Eskandari and Moslehian wrote the main manuscript text and they edited and reviewed the manuscript.

\medskip
\noindent \textit{Conflict of Interest Statement.} On behalf of the authors, the corresponding author states that there is no conflict of interest. 

\medskip
\noindent\textit{Data Availability Statement.} Data sharing is not applicable to this article as no datasets were generated or analysed during the current study.

\medskip
\noindent \textit{Funding Declaration.} No funding was received. 
\medskip

	\end{document}